\documentclass[11pt]{amsart}
\usepackage{graphicx}
\usepackage{setspace}
\usepackage{amssymb,amsthm,amsmath,amstext}
\usepackage{mathrsfs}  
\usepackage{bm}        
\usepackage{mathtools} 
\usepackage{stmaryrd}  

\usepackage[left=1.5in,top=0.95in,right=1.5in,bottom=0.9in]{geometry}

\usepackage[all]{xy}

\theoremstyle{plain}
\newtheorem{theorem}{Theorem}[section]
\newtheorem*{theorem*}{Theorem}

\newtheorem{prop}[theorem]{Proposition}
\newtheorem{lemma}[theorem]{Lemma}
\newtheorem{cor}[theorem]{Corollary}

\newtheorem{comp}[theorem]{Computation}

\newtheorem{theoremintro}{Theorem}

\theoremstyle{definition}
\newtheorem{defin}[theorem]{Definition}

\newtheorem{prob}[theorem]{Problem}

\theoremstyle{remark}
\newtheorem{remark}[theorem]{Remark}

\newcommand{\CH}{\mathrm{CH}}

\newcommand{\isom}{\cong}

\DeclareMathOperator{\rk}{rk}

\renewcommand{\P}{\mathbb{P}}
\newcommand{\Z}{\mathbb{Z}}
\newcommand{\Q}{\mathbb{Q}}
\newcommand{\C}{\mathbb{C}}
\newcommand{\F}{\mathbb{F}}

\newcommand{\CC}{\CC}
\newcommand{\calE}{\mathcal{E}}

\newcommand{\sheaf}[1]{\mathscr{#1}}
\newcommand{\OO}{\sheaf{O}}

\newcommand{\category}[1]{\mathsf{#1}}
\newcommand{\Db}{\category{D}^\textrm{b}}
\renewcommand{\AA}{\category{A}}
\renewcommand{\CC}{\category{C}}

\newcommand{\linedef}[1]{\textsf{#1}}

\usepackage{hyperref}
\hypersetup{pdftitle={Zeta functions of K3 categories over finite fields}}
\hypersetup{pdfauthor={Asher Auel and Jack Petok}}
\hypersetup{colorlinks=true,linkcolor=blue,anchorcolor=blue,citecolor=blue}

\mathchardef\mathcomma=\mathcode`,

\makeatletter
\@namedef{subjclassname@2020}{\textup{2020} Mathematics Subject Classification}
\makeatother

\begin{document}
\author{Asher Auel}
\author{Jack Petok}

\title{Zeta functions of K3 categories over finite fields}

\address{ Department of Mathematics, 
Dartmouth College, Hanover, New Hampshire \newline \indent \texttt{\it \indent E-mail
address: \tt asher.auel@dartmouth.edu}}~
\address{
Department of Mathematics, 
Colby College, Waterville, Maine \newline \indent \texttt{\it \indent E-mail
address: \tt jpetok@colby.edu}}

\begin{abstract}
We study the arithmetic of the K3 category associated to a cubic
fourfold over a non-algebraically closed field $k$, specifically, the
Galois representation on its $\ell$-adic Mukai realization. For $k$ a
finite field, we define the zeta function of a general K3 category, an
invariant under Fourier--Mukai equivalence that can be used to study
its geometricity.  

We show both how the zeta function can obstruct the geometricity of a
K3 category, as well as fail to detect nongeometricity.

Finally, we study an analogue of Honda--Tate for K3 surfaces and for
K3 categories, and provide a nontrivial restriction on the possible Weil
polynomials of the K3 category of a cubic fourfold.

\end{abstract}

\maketitle
\section*{Introduction}
Given a smooth cubic fourfold \(X \subset \P^5\) over a field $k$,
Kuznetsov~\cite{kuznetsov} has established a semiorthogonal
decomposition of its derived category
\[
\Db(X) = \langle \AA_X, \OO, \OO(1), \OO(2)\rangle.
\]
The admissible subcategory $\AA_X$ is known as the \linedef{K3
  category} of the cubic fourfold, cf.\ \cite{huybrechts}, in light of
Kuznetsov's proof that $\AA_X$ is a Calabi--Yau category of dimension $2$ in the
sense of Kontsevich~\cite{kontsevich:CY} and has the same Hochschild
cohomology as a K3 surface.  Such categories are called
\linedef{noncommutative K3 surfaces}, see
\cite[Section~2.2]{macri_stellari:survey}.  We say that $\AA_X$ is
\linedef{geometric over \(k\)} if there is a $k$-linear equivalence
between \(\AA_X\) and $\Db(S)$ for a K3 surface $S$ defined over $k$.
When $k=\C$, Kuznetsov conjectured that $X$ is rational if and only if
$\AA_X$ is geometric.  As evidence for his conjecture, Kuznetsov
checked it for the known families of rational cubic fourfolds, and
more recent work, e.g., \cite{AHTVA}, \cite{addington_thomas},
\cite{ABBVA}, \cite{BLMNPS}, has established additional cases and
shown the conjecture to be equivalent to a Hodge-theoretic
characterization of rationality for cubic fourfolds advocated by
Hassett and Harris.

In this note, we introduce point counting on \(X\) 
as a tool to study the geometricity of the K3 category
\(\AA_X\). Specifically, for $X$ defined over a finite field $k$, we
define the notion of \linedef{point count} \(|\AA_X(k)|\) of the
K3-category $\AA_X$, which is an integer that is a derived invariant
of $\AA_X$ under Fourier--Mukai equivalence.

In the geometric case, when

$\AA_X \cong \Db(S)$ for a K3 surface $S$ defined over $k$, we recover
the classical point count \(|\AA_X(k)| = |S(k)|\) of $S$.  But in
general, the point count of $\AA_X$ may be negative or fail to satisfy
other necessary growth conditions on point counts of varieties.

More generally, we define point counts and zeta functions for any
noncommutative surface, see~Section~\ref{sec:defs}.  Our main motivation is
that point counts can obstruct geometricity of \(\AA_X\). Indeed, when
$X$ is a cubic fourfold over $\Q$ with good reduction at $p$, such
that $\AA_{X_{\F_p}}$ is not geometric over $\F_p$, then \(\AA_X\) is
not geometric over \(\Z_p\), i.e., there is no $\Z_p$-linear
equivalence $\AA_{\mathcal{X}} \cong \Db(\mathcal{S})$ for smooth
proper models $\mathcal{X}$ of $X$ and $\mathcal{S}$ of $S$ over
$\Z_p$.  In this setting, it is expected that any
Fourier--Mukai equivalence $\AA_X \cong \Db(S)$ over $\Q$ spreads to a
Fourier--Mukai equivalence over $\Z_p$ as long as $S$ has good
reduction at $p$, see \cite[Desideratum 3.4.4]{magni:thesis}.  This
would imply that whenever $\AA_{X_{\F_p}}$ is not geometric over
$\F_p$, e.g., has negative point counts, then $X$ does not admit any
associated K3 surface over $\Q$ with good reduction at $p$.

Using the census of cubic fourfolds over $\F_2$ obtained in
\cite{cubiccensus}, we can give a lower bound on the number of
isomorphism classes of cubic fourfolds $X$ over $\F_2$ whose K3
category is not geometric, either because they have negative point counts
(i.e., $|\AA_X(\F_{2^n})| < 0$ for some $n \geq 1$) or because they fail
field extension growth conditions (i.e., $|\AA_X(\F_{2^{mn}})| <
|\AA_X(\F_{2^n})|$ for some $m,n \geq 1$).

\begin{theoremintro}
Of the $1\, 069\, 562$ isomorphism classes of smooth cubic fourfolds
over~$\F_2$, we find that:
\begin{enumerate}
\item[(a)] $2662$ have K3 categories with negative point counts; of
  these, $436$ are ordinary and exactly one is Noether--Lefschetz general.

\item[(b)] $2343$ have K3 categories with nonnegative point counts but
  fail the field extension growth condition; of these, $1084$ are
  ordinary and $140$ are Noether--Lefschetz general. 
\end{enumerate}
In particular, only \(0.47\% \) of smooth cubic fourfolds over
\(\F_2\) have nongeometric K3 category because they have either
negative point counts or fail field extension growth.
\end{theoremintro}

 \noindent Addington informed us that he already observed the
existence of K3 categories with negative point counts in the course of
computer experiments for \cite{addington_auel}.

On the other hand, we expect that point counting alone can sometimes fail
to obstruct geometricity. As evidence, we show the following (see Theorem~\ref{thm:specialexample}).

\begin{theoremintro}\label{thm:thespecialexample}
There exist special cubic fourfolds $X$ over $\Q$ such that:
\begin{itemize}
\item $X$ has good reduction at $2$ and $\AA_{X_{\F_2}}$ has all
positive point counts with field extension growth conditions, and

\item $\AA_{X_\C}$ is not equivalent to $\Db(S,\alpha)$ for any K3
surface $S$ defined over $\C$ and any Brauer class $\alpha \in \mathrm{Br}(S)$.
\end{itemize}
\end{theoremintro}

\noindent In addition, about \(99.87\%\) of Noether--Lefschetz general
cubic fourfold over \(\F_2\) satisfy the conditions in
Theorem~\ref{thm:thespecialexample} (see~\cite[Theorem
1.4]{huybrechts}). The existence of special such cubic fourfolds
indicate that a potential Honda--Tate theory for K3 surfaces (see
Section~\ref{sec:Honda-Tate}) is still quite mysterious, as our current
necessary conditions (e.g., in \cite{kedlaya_sutherland-census}) on
the zeta function of a K3 surface hold for such potentially
noncommutative examples.

Our notion of point count of a K3 category may have additional
applications in the formulation of a Honda-Tate theory for
noncommutative K3 surfaces. As a demonstration of the potential
applications, we first note that if \(\AA_X\) is the K3 category of a
cubic fourfold \(X\), then the categorical Hilbert square
\(\AA_X^{[2]}\) has point-counts arising from geometry: there is an
equality of zeta functions of $\AA_X^{[2]}$ and of the Fano variety
$F(X)$ of lines of a cubic fourfold $X$ over a finite field. We
provide a proof of this (see Proposition~\ref{prop:noncommhilb}),
though it is also deducible from other work (see Remark~\ref{rem:Hilbert_square}).
The geometricity of the Hilbert square gives a nontrivial condition on
whether a zeta function can come from the K3 category of a cubic
fourfold. Indeed, of the \(2\, 971\, 182\) polynomials in~\cite[Computation
3a]{kedlaya_sutherland-census} that are potentially the Weil
polynomial of a noncommutative K3 surface over \(\F_2\), there are
$31\, 256$ that cannot arise from a cubic fourfold because of failure of
field-extension growth of their Hilbert square point counts. See
Section~\ref{sec:Honda-Tate} for a more detailed discussion of
this Honda-Tate theory.

Several authors have considered derived categories and point counts
over finite fields. In~\cite[Conjecture~1]{orlov}, Orlov conjectured
that derived equivalent varieties have the same effective Chow motive,
and hence should have the same point counts. For any variety with
ample or anti-ample canonical class, derived equivalence implies
isomorphism~\cite{BO}, hence the zeta function is a derived
invariant. Several other cases of derived invariance of point counts
have been established: by Antieau, Krashen, and Ward~\cite{akw} for
curves of genus 1, by Lieblich and Olsson~\cite{lieblich_olsson} for
K3 surfaces, and by Honigs~\cite{ honigs1, honigs2} for all abelian
varieties and all smooth projective surfaces and threefolds. This
derived invariance of point counts is especially surprising for
threefolds, since the Hodge numbers of a threefold are not a derived
invariant in positive characteristic~\cite{addington-bragg}.

The key to proving the derived invariance of point counts in these
known cases is to consider the action of Frobenius on the
\(\ell\)-adic Mukai realization \(\bigoplus_i
H^{i}_{\text{\'et}}(\overline{X}, \Q_\ell(\lceil \frac{i}{2} \rceil))
\), a Galois module that only depends on the derived equivalence class
of \(X\). More specifically, let \(k\) be a finite field with \(q\)
elements, let \(X\) be a smooth projective \(k\)-variety, and let
\(\overline{X} =X_{\overline{k}}\). Let \(F \colon \overline{X} \to
\overline{X}\) denote the relative \(q\)-power Frobenius.  Since \(X\)
is smooth and projective, the action of \(F^*\) on \'etale cohomology
satisfies the Grothendieck--Lefschetz trace formula
\[
|X(\F_q)| = \sum_{i=0}^{2\dim X} (-1)^i \mathrm{tr} (F^*\mid
H^i_{\text{\'et}}(\overline{X}, \Q_\ell)).
\]
The cohomology of low dimensional varieties is simple enough that one
can extract, from the action of \(F^*\) on the $\ell$-adic
Mukai realization, just enough information about the eigenvalues of
Frobenius on the individual \(H^i\) to conclude derived invariance of
point counts. Our definition of the point counts of a K3 category is
inspired by this idea.

In this note, we first introduce the notion of point counts of a
Calabi--Yau category of dimension $2$ in Section~\ref{sec:defs}. We
prove some basic properties, including Fourier--Mukai invariance. Then, in
Section~\ref{sec:cubics}, we apply our notion to the study of cubic
fourfolds and their associated noncommutative K3s. We give examples of
noncommutative K3s defined over \(\F_2\) with negative point counts,
which is an obstruction to geometricity of the K3 category. We also
consider the relationship between the Fano variety of lines on the
cubic and the Hilbert square of its K3 category, finding
that these point counts will always agree. Finally, in
Section~\ref{sec:nonadmissible}, we give an example of a nonadmissible
special cubic fourfold over \(\Q\) whose associated noncommutative K3, when reduced modulo 2, has the zeta function of a usual K3 surface. Our example illustrates that the
zeta function is too coarse as invariant to distinguish admissible and
nonadmissible cubics.

\medskip
The authors wish to thank Nick Addington, Sarah Frei, Brendan Hassett, Richard Haburcak, Daniel Huybrechts, Bruno Kahn, Alex Perry, Laura Pertusi, Franco Rota, John Voight, and Xiaolei Zhao. The first author received partial support from Simons Foundation grant 712097, National Science Foundation grant DMS-2200845, and a Walter and Constance Burke Award and a Senior Faculty Grant from Dartmouth College. Part of this work was completed while the second author was a guest researcher at the Junior Trimester Program in Algebraic Geometry at the Hausdorff Institute for Mathematics, funded by the Deutsche Forschungsgemeinschaft (DFG, German Research Foundation) under Germany's Excellence Strategy---EXC-2047/1---390685813.

\section{Galois modules from admissible subcategories}\label{sec:defs}

Let \(X\) be a smooth projective variety defined over a perfect field \(k\) and \(\ell\) a prime not equal to the characteristic of \(k\). Let \(\iota_\CC \colon \CC \hookrightarrow \Db(X)\) denote the embedding of some admissible subcategory of the bounded derived category of \(X\),  and denote by \(\pi_{\CC} \colon \Db(X) \to \CC\) the left adjoint of \(\iota_\CC\). Suppose further that this functor \(\iota_\CC\) is \(k\)-linear. From the derived category of \(X\), one can recover the even and odd Mukai structures:

\[\widetilde{H}^{\text{even}} (X):= \bigoplus_i H^{2i}_{\text{\'et}}(\overline{X}, \Q_\ell(i)), \ \widetilde{H}^{\text{odd}}(X) :=  \bigoplus_i H^{2i-1}_{\text{\'et}}(\overline{X}, \Q_\ell(i)) \]
and the total Mukai structure
\[\widetilde{H}(X) := \widetilde{H}^{\text{even}} (X) \oplus \widetilde{H}^{\text{odd}}(X) .\]

\begin{defin} 
We define the \linedef{\(\ell\)-adic Galois module} of a \(k\)-linear admissible embedding of a subcategory \(\CC \hookrightarrow \Db(X)\)  to be the \(G_k\)-submodule of \(\widetilde{H}(X)\) which is the image of the cohomological Fourier--Mukai transform \((\iota_\CC \circ \pi_{\CC})^H \colon \widetilde{H}(X) \to \widetilde{H}(X)\). We denote this submodule by  \(\widetilde{H}(\CC)\).

 \end{defin}

\begin{remark}
We have defined \(\widetilde{H}(C)\) as a \(\Q_\ell\)-vector space; we do not concern ourselves with defining an underlying integral structure in the present work.
\end{remark}

\begin{remark}
Note that the above definition depends {\it a priori} on the embedding
\(\CC \hookrightarrow \Db(X)\). We will often fix an embedding
\(\CC \hookrightarrow \Db(X)\) and refer to
\(\widetilde{H}(\CC)\) as simply the Galois module associated to
\(\CC\), with the embedding implicitly understood. However, we will
soon see that \(\widetilde{H}(\CC)\) only depends on the embedding
\(\CC \to \Db(X)\) up to Fourier--Mukai equivalence of \(\CC\).
The $\ell$-adic realization functor on noncommutative
motives, see e.g., \cite[Section~3.7]{blanc:realizations}, should
provide an alternative way of constructing \(\widetilde{H}(\CC)\)
independent of the embedding.  
\end{remark}

\begin{defin}
Given a \(\Q_\ell\) vector space \(V\) with an action of a linear operator \(\varphi\), we define the associated \linedef{\(L\)-polynomial} by
\[L(T; \varphi) := \det(\mathrm{Id}- \varphi T)  .\]
Given an admissible embedding \(\CC  \hookrightarrow \Db(X)\), we define the \(L\)-polynomial 
\[ L_{\widetilde{H}(\CC)}(T) := L(T ; F^*).\]
associated to the action of \(F^*\) on \(\widetilde{H}(\CC)\).

\end{defin}

\subsection{Point counts on a noncommutative K3 surface}
\label{subsec:noncommutative_surface}
If one views \(\CC\) as a noncommutative variety, is there a meaningful definition for \(\CC(\F_q)\) or \(|\CC(\F_q)|\)? Unfortunately, there is not in general an obvious way to compute Lefschetz traces from \(\widetilde{H}(\CC)\), because the cohomological Fourier--Mukai transform in general does not preserve cohomological degree. We still manage to define a notion of ``point count'' below, and this notion will agree with classical point counts when \(\CC\) is the bounded derived category of a K3 surface.

\medskip

For the rest of this section, let \(k\) be a field with \(q\) elements
and \(X\) a smooth projective variety over \(k\). Let
\( \CC \hookrightarrow \Db(X)\) be an admissible subcategory defined
over a finite field \(k\). Assume \(\CC\) is a Calabi--Yau category of
dimension $2$ (so the shift $[2]$ is a Serre functor) with the same Hochschild
cohomology as the bounded derived category of a K3 surface, what we
often call a \linedef{noncommutative K3 surface}, see
\cite[Section~2.2]{macri_stellari:survey}.  

\begin{defin}
The \linedef{zeta function} of a noncommutative K3 surface \(\CC\), denoted \(Z_\CC(T)\), is the rational function 

\[ Z_\CC(T) :=  \frac{(1-qT)^2 }{(1-T) \cdot L_{\widetilde{H}(\CC)}(qT) \cdot (1-q^2T)}.\] 

For \(K \supseteq k\) a field extension of degree \(n\), define the \linedef{\(K\)-point count of \(\CC\)} by
 \[|\CC(K)| := na_n, \]
  where \(a_n\) is the coefficient of \(T^n\) in the formal series expansion of \(\log(Z_{\CC}(T))\).
 
\end{defin}
These point counts are rational numbers, but are not necessarily integral or positive.

\begin{remark}
If we further assume that $\CC$ is an \linedef{$\ell$-adic
  noncommutative K3 surface}, i.e., $\widetilde{H}(\CC)$ admits a decomposition
$H^0 \oplus H^2 \oplus H^4$, where $H^0 = H^4 =\Q_\ell$ have trivial
Galois action, then \(L_{\widetilde{H}(\CC)}(T)\) has two factors of $(1-T)$ corresponding
to $H^0$ and $H^4$, and $L_\CC(T) := L_{\widetilde{H}(\CC)}/(1-T)^2$ is a
polynomial that corresponds
to the $H^2$.  This is the case for the K3 category $\AA_X$ of a cubic
fourfold $X$ since $\AA_X$ is the semiorthogonal complement of an
exceptional collection, and we wonder whether it holds for every
noncommutative K3 surface. 
\end{remark}

 \subsection{FM-invariance}
By a \linedef{Fourier--Mukai equivalence} between admissible subcategories \(\CC \hookrightarrow \Db(X)\) and \(\CC' \hookrightarrow \Db(X')\), we mean an exact equivalence \(\CC \cong \CC'\) which appears in the factorization of a Fourier--Mukai transform between \(X\) and \(X'\):
 \[\Db(X) \xrightarrow{\pi_C} \CC \xrightarrow{\sim} \CC' \xrightarrow{\iota_{C'}} \Db(X').\]
Our definition of point count should be invariant under such
equivalences. This would be automatic if \(\CC \to \CC'\) were
induced from an equivalence \(\Db(X) \to \Db(X')\), but in general
there can be  equivalences of admissible subcategories that don't
arise from an equivalence of their ambient derived categories. Still,
we can show that {\it Fourier--Mukai} equivalences induce isomorphisms  \(\widetilde{H}(\CC_X) \to \widetilde{H}(\CC_{X'})\).

\begin{prop} \label{prop:FMinvariance}
If  \(\CC \to \CC'\) is a \(k\)-linear Fourier--Mukai equivalence, induced by some kernel \(\calE \in \Db(X\times X')\) defined over \(k\), then \(\widetilde{H}(\CC) \to \widetilde{H}(\CC') \) is an isomorphism of Galois modules. 
\end{prop}

\begin{proof}
The FM-equivalence \(F \colon \CC \to \CC' \) fits into the sequence of exact functors  
\[\Db(X) \xrightarrow{\pi_X} \CC \xrightarrow{F} \CC' \xrightarrow{\iota_{X'}} \Db(X')\] whose overall composition is Fourier--Mukai.
Let \(\calE \in \Db(X \times X')\) be the kernel of the composition, inducing the exact functor \(\Phi_E\colon  \Db(X) \to \Db(X')\), which by hypothesis is defined over \(k\). Then the right adjoint functor \( \Phi_{\calE_R} \colon \Db(X) \to \Db(X')\) induces an inverse \(k\)-linear equivalence \(\CC' \to \CC\). The Fourier--Mukai kernels induce maps \(\Phi^H_\calE \colon \widetilde{H}(X) \to \widetilde{H}(X')\) and \(\Phi^H_{\calE_R} \colon \widetilde{H}(X') \to \widetilde{H}(X)\); both of these kernels are defined over \(k\) and are therefore compatible with the action of Frobenius, in the sense that \(F^* \calE = \calE\). Since \(\Phi_{\calE} \circ \Phi_{\calE_R}\) restricted to \(\CC\) is naturally isomorphic to the identity on \(\CC\), then we have that the cohomological transform \(\Phi^H_\calE \circ \Phi^H_{\calE_R} \colon \widetilde{H}(X) \to \widetilde{H}(X)\) is compatible with Froebnius and acts as the identity on \(\widetilde{H}(\CC)\), and thus \(\Phi^H_{\calE}\) restricted to \(\widetilde{H}(\CC)\) yields an isomorphism of Galois modules \(\widetilde{H}(\CC) \to \widetilde{H}(\CC')\).
\end{proof}

Since our definition of \(|\CC(\F_q)|\) depends only on the Galois module \(\widetilde{H}(\CC) \), we conclude the FM-invariance of point counts.

\begin{cor}
Let \(k\) be a finite field. If there is \(k\)-linear exact equivalence \(\CC \to \CC'\) which is Fourier--Mukai, then \(Z_{\CC} = Z_{\CC'}\) and \(|\CC(K)| = |\CC'(K)|\) for any finite extension \(K \supseteq k\).
\end{cor}

\subsection{Geometricity of K3 categories}
\label{subsec:geometricity}
Point-counting can help detect whether \(\CC\) is derived equivalent to a (twisted) K3 surface over \(k\). 

\begin{cor} \label{geometriccase}
Let \(k\) be a finite field. If \(\CC \cong \Db(S, \alpha)\) is a \(k\)-linear FM equivalence for some twisted K3 surface \((S, \alpha)\) defined over \(k\), then \(Z_{\CC} = Z_S\); in particular, \(|\CC(K)| = |S(K)|\) for any finite extension \(K \supseteq k\). 
\end{cor}

\begin{proof}
As in the proof of Proposition~\ref{prop:FMinvariance}, there is an isomorphism of Galois modules \(\widetilde{H}(\CC) \cong \widetilde{H}(S, \alpha)\), and since these are \(\Q_\ell\)-vector spaces we further have an isomorphism of Galois modules \(\widetilde{H}(S, \alpha) = \widetilde{H}(S)\), from which one recovers the point counts of \(S\)~\cite{honigs1}.
\end{proof}

The main application of the above corollary is to obstruct
geometricity of the Calabi--Yau category \(\CC\): if the point counts are nonintegral or negative, then \(\CC\) cannot be derived equivalent over \(k\) to a (twisted) K3 surface.

\section{The K3 category of a cubic fourfold}\label{sec:cubics}

In the case where \(X\) is a cubic fourfold defined over a finite
field \(k\) with \(q\) elements, and \(\AA_X \hookrightarrow \Db(X)\)
is the admissible embedding of its K3 category  into its derived category, we wish to study the point counts of \(\AA_X\).

Let \(L_4(T)=\det(\text{Id} - T F^* \mid
H^4_{\text{\'et}}(\overline{X}, \Q_\ell(2)) ) \) denote the
\(L\)-polynomial on the middle cohomology of $X$, which is a
polynomial of degree 23 whose roots have complex absolute value 1. Let
\(L_{4,pr}(T)= L_4(T)/(1-T)\) denote the factor corresponding to the
primitive cohomology. Then we have
\[Z_{\AA_X}(T) = \frac{1-qT}{(1-T)L_4(qT)(1-q^2T)} =\frac{1}{(1-T)L_{4,pr}(qT)(1-q^2T)}.\]
We remark that this gives the following formula for the point counts of $\AA_X$ in
terms of the point counts of $X$
\[
|\AA_X(\F_{q^n})| = \frac{|X(\F_{q^n})| - 1 - q^{2n} -q^{4n}}{q^n}
\]
and inversely, a formula for the point counts of $X$ in terms of those
for $\AA_X$
\begin{equation}
\label{eq:cubic_fourfold_point_counts}
|X(\F_{q^n})| =1 +q^n|\AA_X(\F_{q^n})| + q^{2n}   + q^{4n} 
\end{equation}
These formulae, together with Chevalley--Warning--Ax theorem~\cite{ax}, which
gives that $|X(\F_q)| \equiv 1 \pmod q$, imply the following.

\begin{lemma}
Let \(\AA_X \hookrightarrow \Db(X)\) be the K3 category of a cubic
fourfold $X$ over a finite field \(k\), and let \(K/k \) be a finite extension. Then the point count \(|\AA_X(K)|\) is an integer. $\qed$
\end{lemma}

\begin{remark}
Li, Pertusi, and Zhao~\cite{LPZ} show, at least over the complex numbers,
that any exact equivalence of K3 categories of cubic fourfolds is
Fourier--Mukai. If one could show this over finite fields, then
consequently the point counts of \(\AA_X\) would be independent of the
admissible embedding. The missing ingredient over a finite field is
the nonemptiness of a certain moduli space of objects in the K3 category, but this would take us too far afield from our primary focus.
\end{remark}

\subsection{Negative point counts and obstructions to geometricity}

In the database~\cite{github}, we have examples of cubic fourfolds
over \(\F_2\) whose K3 categories have some negative point counts \(|\AA_X(\F_2)| < 0\). This could be indicative of some associated K3 defined over some larger base extension, or it could indicate that there is no associated K3 even geometrically. In any case, we can definitively say that these cubics have no associated K3 defined over \(\F_2\).

\begin{comp}\label{comp:F2negptcounts}
There are \(2662\) cubic fourfolds up to isomorphism over \(\F_2\) for which \(|\AA_X(K)|<0\) for some finite extension of \(K\), and hence do not have associated K3 over the field \(\F_2\). 
\end{comp}

In fact, all but one of these cubics is Noether--Lefschetz special, demonstrating that negative point counts can be used to exhibit explicit special cubic fourfolds with no associated K3 over \(\F_2\). 

\begin{comp} \label{negptcountsonvgeneral}
There is a unique $\F_2$-isomorphism class of smooth cubic fourfold
$X$ defined over \(\F_2\) that is geometrically Noether--Lefschetz general and whose
K3 category \(\AA_X\) has negative point counts, represented by
\[
x_1^2x_2 + x_1^2x_6 + x_1x_2x_6 + x_1x_3x_5 + x_1x_4^2 + 
    x_1x_5^2 + x_1x_6^2 + x_2^3 + x_2^2x_5 + x_2^2x_6 \]
   \[+ x_2x_3x_4 + x_2x_5^2 + x_3^3 + x_3x_6^2 + x_4^2x_5 + 
    x_4x_6^2 + x_6^3.
\]
\end{comp}

\medskip

\begin{remark}
There are 436 among those from Computation~\ref{comp:F2negptcounts} which are ordinary. The cubic fourfold of Computation~\ref{negptcountsonvgeneral} is a non-ordinary cubic fourfold (of height~7).

\end{remark}

\begin{remark}
 We also found that there are a handful of cubics for which the \(\F_2\)-point counts of \(\AA_X\) are positive but for which \(|\AA_X(\F_{2^m})| <0\) for some \(m > 1\). 
\end{remark}

Another natural condition on the point counts for a K3 surface \(S\),
considered by Kedlaya and Sutherland~\cite{kedlaya_sutherland-census}, is that \(|S(\F_{q^{mn}})| \ge |S(\F_{q^n})| \) for all \(m , n \ge 1\).
\medskip 

This provides, for a cubic fourfold \(X\), an obstruction to geometricity over \(k\).

\begin{comp} There are 2343 cubic fourfolds over \(\F_2\) that have \(|\AA_X(\F_{2^k})| >0\) for all \(k \ge 1\) but \(|\AA_X(\F_{2^{mn}})| < |\AA_X(\F_{2^{n}})|\) for some \(m,n \ge 1\), and hence do not have associated K3 over \(\F_2\).
\end{comp}

\begin{remark} 
There are 1084 cubic fourfolds in the above computation which are
ordinary. 
\end{remark}

\subsection{Hilbert schemes and Fano varieties}

For \(X\) a smooth cubic fourfold defined over the finite field \(k=\F_q\) and for any finite extension \(\F_{q^n} \supset k\), the point counts of \(X\) determine the point counts of its Fano variety of lines \(F_1(X)\). Precisely, we have~(\cite[Corollary 5.2]{galkin_shinder},~\cite[Equation 8]{debarre_laface_roulleau}):

\begin{equation}\label{fanocount}
|F_1(X)(\F_{q^n}) | = \frac{|X(\F_{q^n})|^2 -2(1+q^{4n})|X(\F_{q^n})| +|X(\F_{q^{2n}})| } {2q^{2n}}. 
\end{equation}

On the other hand, if \(S^{[2]}\) is the Hilbert scheme of length two subscheme on a K3 surface we have a formula for its point counts, which was first used by G\"{o}ttsche in~\cite{gottsche}:
\[
|S^{[2]}(\F_{q^n})| = \binom{|S(\F_{q^n})|}{2} + (q^n+1)|S(\F_{q^n})| + \frac{|S(\F_{q^{2n}})| - |S(\F_{q^n})|}{2}
\]

Ganter and Kapranov~\cite{ganter_kapranov} defined the symmetric
square $\CC^{[2]}$ of any $k$-linear triangulated (or dg-)category
$\CC$.  When $S$ is a smooth projective surface, one has a $k$-linear
equivalence $\Db(S)^{[2]} \to \Db(S^{[2]})$, see
\cite[Chapter~7,~Remark~3.28(i)]{huybrechtscubicbook}.  This formally
motivates the following definition of point counts of the Hilbert square
\(\AA_X^{[2]}\) of the K3 category (or any noncommutative K3) 
\begin{equation}
\label{eq:hilbert_square_point_counts}
|\AA_X^{[2]}(\F_{q^n})| := \binom{|\AA_X(\F_{q^n})|}{2} + (q^n+1)|\AA_X(\F_{q^n})| + \frac{|\AA_X(\F_{q^{2n}})| - |\AA_X(\F_{q^n})|}{2}
\end{equation}
as well as the corresponding zeta function $Z_{\AA_X^{[2]}}(T)$.

\begin{prop}\label{prop:noncommhilb}
Let \(X\) be a cubic fourfold over a finite field \(k\). Then there is a equality of zeta functions
\[
Z_{\AA_X^{[2]}} (T) = Z_{F_1(X)}(T)
\]
\end{prop}

\begin{proof}
By the Weil conjectures, there are algebraic integers \(\alpha_1, \ldots, \alpha_{22}\) such that
\begin{equation}\label{eq:Xq}
|X(\F_{q^n})|  = 1 + q^n + \sum_{j=1}^{22}\alpha_j^n + q^{2n} + q^{3n}
+ q^{4n},
\end{equation} 
\[|X(\F_{q^{2n}})|  = 1 + q^{2n} + \sum_{j=1}^{22}\alpha_j^{2n} + q^{4n} + q^{6n} + q^{8n}.\] 
For the K3 category, one has 
\[|\AA_X(\F_{q^n})| = 1 + \sum_{j=1}^{22} \left( \frac{\alpha_j}{q}\right)^n + q^{2n}, \]
\[|\AA_X(\F_{q^{2n}})| = 1 + \sum_{j=1}^{22} \left( \frac{\alpha_j}{q}\right)^{2n} + q^{4n}. \]
Evaluating these expressions into Equations~\ref{fanocount}
and~\ref{eq:hilbert_square_point_counts}, one find formally that
\[ 
|\AA_X^{[2]}(\F_{q^n})| = |F_1(X)(\F_{q^n})| 
\]
for all $n \geq 1$ and thereby giving the equality of zeta functions.

\end{proof}

\begin{remark}
\label{rem:Hilbert_square}
The construction of the symmetric square implies that \(\AA_X^{[2]}
\subset \Db(X^{[2]})\) is an admissible subcategory, where $X^{[2]}$
is the Hilbert scheme of length $2$ subschemes on $X$, see
\cite[Chapter~7,~Remark~3.28(ii)]{huybrechtscubicbook}.  Work of
Belmans, Fu, and Raedschelders~\cite[Theorem~B]{bfr} (cf.\
\cite[Section~8]{huybrechtsupdate}) presents both $\AA_X^{[2]}$ and
$\Db(F_1(X))$ as pieces of a semiorthogonal decomposition of
$\Db(X^{[2]})$ with all other components equivalent to $\Db(X)$, from which
one can also deduce the equality of
\(L\)-polynomials \(L_{\widetilde{H}(F_1(X))}(T) = L_{\widetilde {H}(\AA_X^{[2]})}(T) \), thereby providing a different
proof of Proposition~\ref{prop:noncommhilb}.

It would be interesting to consider an appropriate Grothendieck ring of
noncommutative varieties, which is a generalization of both the
Grothendieck ring of $k$-varieties and the Grothendieck ring of differential graded $k$-linear
categories, in which the formula (in analogy with
work of Galkin and Shinder~\cite[Theorem~5.1]{galkin_shinder}, see \cite[Chapter~7,~Remark~3.28(ii)]{huybrechtscubicbook})
\[
[X^{[2]}] = \mathbb{L}^2\, [\AA_X^{[2]}] + [\P^4][X]
\]
could reside, where $\mathbb{L} = [\mathbb{A}^1]$.  One can check directly, in analogy with the proof of
Proposition~\ref{prop:noncommhilb}, that if $k$ is a finite field, then
\[
|\AA_X^{[2]}(k)| = \frac{|X^{[2]}(k)| - |\P^4(k)||X(k)|}{q^2},
\]
an equality that would follow from the existence of an appropriate
point counting motivic measure on this Grothendieck ring.  One could
also ask whether the equality $[F_1(X)] = [\AA_X^{[2]}]$ holds in this
ring (and not just after inverting $\mathbb{L}$).  

Finally, whether there is actually an equivalence
$\AA_X^{[2]} \isom \Db(F_1(X))$, which would immediately imply Proposition~\ref{prop:noncommhilb}, is a conjecture attributed to Galkin,
and still open in general, see
\cite[Chapter~7,~Remark~3.28(iii)]{huybrechtscubicbook} and
\cite[Section~8]{huybrechtsupdate}.
\end{remark}

\begin{remark}
 Frei~\cite{frei} has already shown that point counts of moduli spaces
 of stable sheaves on K3 surfaces only depend on the dimension of the
 moduli space and the underlying K3. Li, Pertusi, and
 Zhao~\cite{LPZfano} have shown that \(F_1(X)\) is a moduli space of
 Bridgeland stable objects in $\AA_X$. Therefore, our
 Proposition~\ref{prop:noncommhilb} is evidence that Frei's result
 should extend to the noncommutative setting: the point count of a
 moduli space of Bridgeland stable objects on a noncommutative K3
 surface should only depend on the category and the dimension of the
 moduli space. 
\end{remark}

\section{Towards a Honda--Tate for K3 surfaces}
\label{sec:Honda-Tate}

Kedlaya and Sutherland~\cite{kedlaya_sutherland-census} have initiated
a program that could be called \linedef{Honda--Tate for K3 surfaces},
in analogy with the classical Honda--Tate theorem \cite{honda,tate}
for abelian varieties over a finite field.  In this section, we give
some further details of what this program entails and provide a few
additional observations, including how varieties ``of K3-type'' (such
as cubic fourfolds) might play a role.

For context, we will review the classical Honda--Tate theorem in the
case of abelian varieties.  For an abelian variety $A$ of dimension
$g$ over $\F_q$, the zeta function $\zeta_A(T)$ is completely
determined by, and determines, the characteristic polynomial
$\Phi_A(T)$ of Frobenius on $H^1_{\text{\'et}}(\overline{A},\Q_\ell)$
for any $\ell$ prime to $q$.  In this case, $\Phi_A(T)$ is a Weil
polynomial of degree $2g$, all of whose roots have absolute value
$q^{1/2}$ and must satisfy other arithmetic conditions.  Isogenous
abelian varieties have the same zeta function, and the map from
the set of isogeny classes of abelian varieties of dimension $g$ to the set of Weil polynomials is injective by a result of Tate~\cite{tate}.  The
description of the image of the map, i.e., the set of Weil polynomials
realized by abelian varieties of dimension $g$, is a result of
Honda~\cite{honda}.

For a K3 surface $S$ over $\F_q$, the zeta function is given by
$$
\zeta_S(T) = \frac{1}{(1-T)L_S(qT)(1-q^2t)}
$$ 
where $L_S(T)$ is the $L$-polynomial of Frobenius acting
$H^2_{\text{\'et}}(\overline{S},\Q_\ell(1))$ for any $\ell$ prime to
$q$.  In this case, $L_S(T)$ is a Weil polynomial of degree $22$, all
of whose roots have absolute value $1$ and must satisfy other
constraints, see Theorem~\ref{thm:Taelman_conditions}.

There is a classical notion of \linedef{isogeny} of K3 surfaces $S$
and $S'$ over $\C$, namely, an isometry
$\varphi \colon H^2(S(\C),\Q) \to H^2(S'(\C),\Q)$ of rational Hodge
structures with intersection pairing.  Various authors have proposed
notions of isogeny between K3 surfaces over a finite field, e.g.,
\cite{bragg_yang} and \cite{yang}, which would produce
an isometry
$\varphi \colon H^2_{\text{\'et}}(\overline{S},\Q_\ell) \to
H^2_{\text{\'et}}(\overline{S}',\Q_\ell)$ of Galois modules with
intersection pairing for all $\ell$ prime to $q$.  In particular,
isogenous K3 surfaces should have the same Weil polynomial.
Honda--Tate for K3 surfaces then consists of the following two problems.

\begin{prob}[Honda--Tate for K3s]
$\phantom{1}$
\begin{enumerate}
\item ``Tate for K3s''
Determine whether the map from isogeny classes of K3 surfaces over $\F_q$ to Weil
 polynomials is injective.

\item ``Honda for K3s''
Determine the Weil polynomials that arise from K3 surfaces over $\F_q$.
\end{enumerate}
\end{prob}

Tate for K3s is likely known to the experts and should follow from the
semisimplicity of Frobenius acting on $\ell$-adic cohomology, whereas
Honda for K3s seems to be wide open, and is the subject of Kedlaya and
Sutherland's computational work~\cite{kedlaya_sutherland-census}, as
well as work by Taelman~\cite{taelman} and Ito~\cite{ito}.  In
particular, Kedlaya and Sutherland describe an algorithm to generate a
list of all polynomials that could potentially arise as Weil
polynomials of K3 surfaces over a fixed finite field $\F_q$, though it is
still unknown whether all such polynomials do arise from K3 surfaces
defined over $\F_q$.  On the other hand, Taelman and Ito prove
that under mild hypotheses, every Weil polynomial on such a list is
realized by a K3 surface defined over an extension of $\F_q$.

From the Weil conjectures and properties of crystalline cohomology,
the Weil polynomials $L_S(T)$ must satisfy the following arithmetic
constraints, see \cite{kedlaya_sutherland-census,ito,taelman}.

\begin{theorem}
\label{thm:Taelman_conditions}
Let $S$ be a K3 surface over $\F_q$.  Then the Weil polynomial
$L_S(T) \in \Q[T]$ is a degree 22 polynomial, all of whose roots have
complex absolute value 1, and satisfies:
\begin{enumerate}
\item \textit{Projectivity.} $L_S(T)$ has a factor of $1-T$.

\item \textit{Weil conjectures.} $L_S(T) \in \Z_\ell[T]$ for all
  $\ell$ prime to $q$.

\item \textit{Crystalline.} Factor $L_S(T) =
  L_{S,\text{alg}}(T) L_{S,\text{trc}}(T)$ where $L_{S,\text{alg}}(T)$
  is the maximal factor all of whose roots are roots of unity.  Then
  either $S$ is supersingular, in which case $L_S(T) =
  L_{S,\text{alg}}(T)$, or $L_{S,\text{trc}}(T)$ satisfies the following two properties:
\begin{enumerate}
\item \textit{Newton above Hodge.} The Newton polygon of
  $L_{S,\text{trc}}(T) \in \Q_p(T)$ lies above the Hodge polygon of the crystalline transcendental lattice of $S$.

\item \textit{Transcedental.} $L_{S,\text{trc}}(T) = Q^e$ for some
  $e > 0$ and $Q \in \Q[T]$ irreducible, where $Q = Q_{<0} Q_{\geq 0}$
  in $\Q_p[T]$ where $Q_{<0}$ is irreducible and consists of all
  factors of $Q$ of the form $(1-\gamma T)$ with $v_p(\gamma) <0$.
\end{enumerate}
We note that when $q = p$, the conditions $\mathrm{(a)}$ and $\mathrm{(b)}$ are together equivalent to $pL_S(T) \in \Z_p[T]$.

\item \textit{Nonnegative point counts.} $L_S(T)$ is consistent with
  $|S(\F_{q^n})| \geq 0$ for all $n \geq 1$.

\item \textit{Field extension growth.} $L_S(T)$ is consistent with
  $|S(\F_{q^m})| \geq |S(\F_{q^n})|$ for all $m, n \geq 1$ with $n|m$.

\item \textit{Artin--Tate.} Writing $L_S(T) = (1-T)^r L_1(T)$ with
  $L_1(1) \neq 0$, we have $qL_1(-1)$ is a square (possibly
  0).
\end{enumerate} 
\end{theorem}

We say that a degree $22$ Weil polynomial satisfying these conditions
is of \linedef{K3 type} over $\F_q$.    One reasonable
question is whether these conditions provide a complete solution to Honda
for K3s. To test this, Kedlaya and Sutherland in
\cite[Computation~3]{kedlaya_sutherland-census} enumerated the entire
list of Weil polynomials of K3 type\footnote{The
degree 21 polynomial \(L(T) \in \Z[T]\) appearing in
\cite{kedlaya_sutherland-census} is expressed in terms of our degree
22 Weil polynomial as $L(T)=q L_S(T/q)/(1-T)$.} over $\F_2$, finding
1,672,565 such polynomials. They compare this list to the Weil
polynomials of quartic K3 surfaces over \(\F_2\), which account for
only about \(3\%\) of all Weil polynomials of K3 type for \(q=2\).  It
would be a natural next step to create a census K3 surfaces of other
low degree over $\F_2$.

\smallskip

One could also consider a ``Honda--Tate for noncommutative K3s'' for
the larger class of noncommutative K3 surfaces $\CC$ over a finite
field $\F_q$ together with their Weil polynomials $L_\CC(T)$
introduced in Section~\ref{subsec:noncommutative_surface}.  In this
context, formulating ``Tate for noncommutative K3s'' would require a
notion of isogeny between noncommutative K3 surfaces, whose details
are not entirely clear, but should at least imply that isogenous
noncommutative K3 surfaces have the same Weil polynomials.  In order
to resolve ``Honda for noncommutative K3s'' one must formulate a
reasonable list of necessary properties satisfied by Weil polynomials
of noncommutative K3s, and then ask whether these properties are
sufficient.

As a concrete example, we now list some known conditions on the Weil
polynomials of K3 categories of cubic fourfolds, in the case $q=p$ for
simplicity.  To that end, we introduce some formal point counts
associated to a noncommutative K3 surface.  In analogy with
\eqref{eq:cubic_fourfold_point_counts}, we define the \linedef{cubic
fourfold point count} of $\CC$ as
\[
|X_\CC(\F_{q^n})| = 1 +q^n|\AA_X(\F_{q^n})| + q^{2n}   + q^{4n} 
\] 
and in analogy with \eqref{eq:hilbert_square_point_counts} we define
the \linedef{Hilbert square point count} of $\CC$ as
\[
|\CC^{[2]}(\F_{q^n})| = \binom{|\CC(\F_{q^n})|}{2} + (q^n+1)|\CC(\F_{q^n})| + \frac{|\CC(\F_{q^{2n}})| - |\CC(\F_{q^n})|}{2}
\]

\begin{prop}
\label{prop:noncommutative_Taelman}
Let $\CC$ be the K3 category of a cubic fourfold over $\F_p$.
Then the Weil polynomial $L_\CC(T) \in \Q[T]$ is a degree 22
polynomial, all of whose roots have complex absolute value 1, and
satisfies:
\begin{enumerate} 
\item Weil Conjectures and Crystalline: $pL_\CC(T) \in \Z[T]$.

\item \textit{Nonnegative cubic fourfold point counts.} $L_\CC(T)$ is consistent with
  $|X_\CC(\F_{p^n})| \geq 0$ for all $n \geq 1$.

\item \textit{Cubic fourfold field extension growth.} $L_\CC(T)$ is consistent with
  $|X_\CC(\F_{p^m})| \geq |X_\CC(\F_{q^n})|$ for all $m, n \geq 1$ with $n|m$.

\item \textit{Nonnegative Hilbert square point counts.} $L_\CC(T)$ is consistent with
  $|\CC^{[2]}(\F_{p^n})| \geq 0$ for all $n \geq 1$.

\item \textit{Hilbert square field extension growth.} $L_\CC(T)$ is consistent with
  $|\CC^{[2]}(\F_{p^m})| \geq |\CC^{[2]}(\F_{p^n})|$ for all $m, n \geq 1$ with $n|m$.
  
 \item Artin--Tate. Writing $L_\CC(T) = (1-T)^r L_1(T)$ with
  $L_1(1) \neq 0$, we have $pL_1(-1)$ is a square (possibly
  0).
\end{enumerate}
\end{prop}
\begin{proof}
Parts (1), (2), (3) come from being a factor of the cubic fourfold
Weil polynomial.  Parts (4) and (5) comes from
Proposition~\ref{prop:noncommhilb}.  Part (6) comes from the fact that $\AA_X$ is the left orthogonal to a
exceptional collection, so that $L_{\AA_X}(T)$ differs from $L_X(T)$
by factors of $(1-T)$, and then applying \cite[Theorem~1.9]{elsenhans_jahnel:duke} for $p >2$ or \cite[Theorem~4.6]{cubiccensus} for $p=2$. 
\end{proof}

As in the case for K3 surfaces, for a noncommutative K3 surface $\CC$,
we write $L_\CC(T) = L_{\CC,\text{alg}}(T) L_{\CC,\text{trc}}(T)$
where $L_{\CC,\text{alg}}(T)$ is the maximal factor all of whose roots
are roots of unity.  For a K3 surface $S$, the Tate conjecture (which
are proved in \cite{maulik}, \cite{charles1}, \cite{charles2},
\cite{madapusipera}, \cite{kim_madapusipera},
\cite{madapusipera-erratum}, \cite{ito_ito_koshikawa})) implies that
the geometric Picard rank of $S$ equals the degree of
$L_{S\text{alg}}$ and the arithmetic Picard rank (i.e., Picard rank
over $\F_q$) of $S$ equals the multiplicity of $(1-T)$ in
$L_{S,\text{alg}}$.  One important difference in the noncommutative
setting is that $L_{\CC,\text{alg}}(T)$ may equal $1$; we call such
noncommutative K3 surfaces ``purely transcendental.''  Hence by
analogy, purely transcendental noncommutative K3 surfaces have
``geometric Picard rank 0.'' For $q=2$, we can collect some statistics
about the distribution of the \linedef{arithmetic Picard rank} $\rho$
(i.e., the multiplicity in $L_\CC(T)$ of the root 1) and
\linedef{geometric Picard rank} $\overline{\rho}$ (i.e., the degree of
$L_{\CC,\text{alg}}(T)$) among Weil polynomials of degree 22 over
$\F_q$ that satisfy a minimal set of necessary conditions to be
potentially realizable by a noncommutative K3 surface.

\begin{comp}
There are 5,478,058 degree 22 Weil polynomials satisfying the ``Weil
conjecture and crystalline'' condition from
Proposition~\ref{prop:noncommutative_Taelman} over $\F_2$.  Of these,
we record the distribution of (geometric) Picard ranks:
\[
\renewcommand{\arraystretch}{1.1}
\hspace*{-0.8cm}
\begin{array}{c|cccccc}
\overline{\rho} &0 & 2 & 4 & 6 & 8 & 10 \\
\#              &74846 & 242700 & 441072 & 697944 & 762944 & 936736\\[4pt] \hline \\[-9pt]
\overline{\rho} & 12& 14& 16 & 18 & 22 & 22 \\
\# & 775320 & 651600 & 442308 & 270180 & 122128 & 60280 \\[-13pt] 
    \end{array}
\]
\bigskip
\\
\[
\renewcommand{\arraystretch}{1.1}
\hspace*{-0.62cm}
\begin{array}{c|cccccc}
\rho &0 & 1 & 2 & 3 & 4 & 5 \\
\#   &2506876 & 956904 & 956904 & 349118 & 349118 & 121936\\[4pt] \hline \\[-9pt]
\rho &6& 7& 8 & 9 & 10 & 11 \\
\# & 121936 & 40194 & 40194 & 12574 & 12574 & 3612 \\[4pt]
 \hline \\[-9pt]
\rho &12& 13& 14& 15 & 16& 17 \\
\# & 3612 & 966 & 966 & 230 & 230 & 48 \\[4pt]
 \hline \\[-9pt]
\rho &18& 19& 20&21 & 22 \\
\# & 48 & 8 & 8 & 1 & 1 \\[-13pt] 
    \end{array}
\]
\end{comp}
\bigskip

We also remark that of the 74,846 purely transcendental Weil polynomials
above, only 4,294 come from K3 categories of a cubic fourfolds over
\(\F_2\).  

\begin{remark}
Since roots of $L_{\CC,\text{trc}}(T)$ must come in complex
conjugate pairs, the possible values of $\overline{\rho}$ are even and
that for every odd number $n = 1,\dotsc,21$, the number of Weil
polynomials with $\rho = n$ and $\rho= n+1$ is the same.  This is the
same reason why the Tate conjecture implies the well-known fact that
the geometric Picard rank of a K3 surface over a finite field must be
even. 
\end{remark}

\section{The zeta function of a nonadmissible cubic fourfold}\label{sec:nonadmissible}
The goal of this section is give a particular example of a Noether--Lefschetz special cubic fourfold \(X\) over \(\Q\) with nongeometric K3 category. The example suggests that the zeta function alone cannot be used to detect nongeometric K3 categories.  

\begin{theorem}\label{thm:specialexample}
There exists a cubic fourfold \(X/\Z\) such that
\begin{enumerate}
\item \(X_\Q\) has no associated K3 and no associated twisted K3 surface over \(\C\) (i.e., \(X_\Q\) is not admissible nor twisted-admissible over \(\C\));
\item \(X\) has good reduction at \(p=2\) and the specialization \(\CH^2(X_{\overline{\Q}}) \to \CH^2(X_{\overline{\F}_2}) \) is an isomorphism; and 
\item the point counts of \(\AA_{X_{\F_2}}\) satisfy the
conditions of Theorem~\ref{thm:Taelman_conditions}.
\end{enumerate}
\end{theorem}

In light of the description of the point counts of
\(\AA_{X_{\F_2}}\) above, it is still plausible that \(
Z_{\AA_{X_{\F_2}}}(T)\) is the zeta function of an actual
K3 surface over \(\F_2\).  Indeed, the main result of
Taelman~\cite{taelman} and the computational evidence
in~\cite{kedlaya_sutherland-census} suggest that if \(\AA_X\) has
point counts resembling a K3 surface then \(Z_{\AA_X}(T)\) should be equal to the zeta function of an actual K3 surface over \(k\). Therefore, the results of this section can be taken as evidence that point-counting on a K3 category cannot alone distinguish geometric K3 categories from nongeometric ones. 

\subsection{A special cubic fourfold over $\Q$ with no (twisted) admissible markings}\label{ss:noadmissiblemarkings}

Let $d$ be a positive integer congruent to $0$ or $2$ modulo $6$. We say that $d$ is an {\it admissible discriminant} if $d$ satisfies the condition
\[
4 \nmid d, 9 \nmid d, \text{ and } p \nmid d \text{ for any odd prime } p \equiv 2 \mod 3 \ \ (\star \star).
\]
We say that $d$ is a {\it twisted admissible discriminant} if there is some integer $k$ such that 
\[d=k^2d_0 \text{ for some admissible discriminant $d_0$} \ \ (\star \star '). \]

For a cubic fourfold \(X\) over the complex numbers,  Huybrechts~\cite{huybrechts} has shown, following closely-related work of Addington and Thomas~\cite{addington_thomas}, that $\CH^2(X)$ contains a primitive rank $2$ sublattice $K$ of {\it (twisted) admissible discriminant} $d$ with $c_1(\mathcal{O}_X(1) )^2 \in K$ if and only if $\AA_X \cong \Db(S, \alpha)$ for some twisted K3 surface \((S, \alpha)\). We in turn say that the primitive sublattice $K$ is a {\it (twisted) admissible} sublattice.

\begin{prop}\label{prop:noadmissiblemarkings}
If $X$ is a cubic fourfold $X/\C$ with \(\rk\,\CH^2(X) = 3\) and \(X\) contains a Veronese surface $V$ and a cubic scroll $T$ such that $T.V=2$, then $\AA_X$ is nongeometric; that is, $\AA_X$ is not equivalent to the derived category of a (twisted) K3 surface.
\end{prop}

\begin{proof}
Let $h = c_1(\mathcal{O}_X(1))$. By the hypotheses, the cubic fourfold $X$ contains a cubic scroll, so $X$ has a rank $2$ marking of discriminant $12$, and a discriminant $20$ marking as well because it contains a Veronese~\cite{hassett}. Since $T.V=2$,  the lattice $\langle h^2, T, V\rangle \subset \CH^2(X)$ has Gram matrix
\[ \begin{pmatrix}
3 & 4 & 3 \\
4 & 12 & 2          \\
3 & 2 & 7
\end{pmatrix}
\]
We note that in~\cite[\S 8]{yang_yu}, Yang and Yu give a classification of rank $3$ positive definite lattices $M$ that can contain (twisted) admissible primitive sublattices. As a consequence of their classification, $M$ has {\it no admissible} sublattices. We wish to show the stronger claim that $M$ has no {\it twisted admissible} primitive sublattices. Indeed, the discriminant of any rank $2$ primitive sublattice containing $h^2$ of the rank $3$ lattice above is generically of the form $d=12y(y-z) + 20z^2$ for integers $y$ and $z$ and so cannot be admissible (since $4 \mid d$), and further more cannot be twisted admissible: if $d$ were twisted admissible, then $d/4=3y(y-z)+5z^2$ would need to be of the form $s^2d_0$ for some admissible $d_0$, so in particular we would need $d/4 \equiv 0$ or $2$ modulo $6$, but $3y(y-z) + 5z^2$ never represents these congruence classes.
\end{proof}

\begin{remark}
In Proposition~\ref{prop:noadmissiblemarkings}, we are using a very
general cubic fourfold contained in the intersection $\mathcal{C}_{12}
\cap \mathcal{C}_{20}$ of Hassett divisors in the moduli space of
cubic fourfolds.  In fact, it appears that $\mathcal{C}_{12}
\cap \mathcal{C}_{20}$ has six irreducible components, determined by the
possible intersection numbers $T.V = -1, \dotsc, 4$.  We utilize the
component where $T.V=2$, and recently, the geometry of the component
with $T.V = 4$ has been investigated in \cite{prieto}.   
\end{remark}

\begin{proof}[Proof of Theorem~\ref{thm:specialexample}]
Let \(X/\Z\) be the cubic fourfold with equation

\noindent{\tiny
\[-27195x_1^3 + 99309x_1^2x_2 + 52143x_1^2x_3 - 19299x_1^2x_4 + 17717x_1^2x_5 - 166089x_1^2x_6 + 280203x_1x_2^2 + 42138x_1x_2x_3 - 24486x_1x_2x_4 + 335080x_1x_2x_5 \]
\[+ 36287x_1x_2x_6 - 52038x_1x_3^2 + 42628x_1x_3x_4 - 91243x_1x_3x_5 + 76026x_1x_3x_6 
+ 74191x_1x_4^2 + 105644x_1x_4x_5 - 206488x_1x_4x_6 \] \[- 21765x_1x_5^2 - 396946x_1x_5x_6 - 145953x_1x_6^2 + 153699x_2^3 - 17064x_2^2x_3 - 12246x_2^2x_4 + 317363x_2^2x_5 + 202376x_2^2x_6 - 45743x_2x_3^2 + 76777x_2x_3x_4 \]\[- 160450x_2x_3x_5 - 622x_2x_3x_6 
+ 102045x_2x_4^2 + 105638x_2x_4x_5 - 206876x_2x_4x_6 + 104046x_2x_5^2 - 97682x_2x_5x_6 \]\[+ 48677x_2x_6^2 + 27090x_3^3 + 54944x_3^2x_4 - 93628x_3^2x_5- 11594x_3^2x_6 + 27854x_3x_4^2 - 13462x_3x_4x_5 + 97190x_3x_4x_6 - 149681x_3x_5^2 
- 126562x_3x_5x_6 \]\[+ 44296x_3x_6^2 + 74191x_4^2x_5 + 102045x_4^2x_6 + 97089x_4x_5^2 - 76364x_4x_5x_6 - 194630x_4x_6^2 - 52559x_5^3 - 376318x_5^2x_6 - 296287x_5x_6^2 x_1 = 0.\]
}

\noindent We wish to show that \(X\) has the claimed properties (1)
and (2). We first show that \(X_\Q\) contains surfaces \(T\) and \(V\) as in the statement of Proposition~\ref{prop:noadmissiblemarkings}. The cubic \(X\) contains the cubic scroll \(T\) given by the simultaneous vanishing of the \(2 \times 2\) minors of the matrix

\[
 \begin{pmatrix} 
 x_5 & x_2 + x_3+x_6 \\
 x_5+x_6 & x_2+x_6 \\
 x_3+x_4+x_6 & x_5 
 \end{pmatrix},\]
 inside of the hyperplane \(x_1+x_2+x_5+x_6=0\).
 One also computes that \(X\) contains the Veronese surface \(V\) 
 given by the vanishing of the minors of the matrix 
 \[
 \begin{pmatrix}
 x_3+x_5 & x_1+x_2+x_4 & x_3+x_4\\
 x_1+x_2+x_4 & x_1+x_4+x_6 & x_3 + x_6 \\
 x_3 + x_4 & x_3+x_6 & x_1+x_2+x_5+x_6 
 \end{pmatrix}
 \]
The Magma code~\cite{Magma} in the arXiv distribution of this article
verifies these claims and the claim that \(T.V = 2\). We thus know that \(\rk \CH^2(X_\Q) \ge 3\). 

We next verify using Magma that the reduction \(X_{\F_2} \) is
smooth. Now, for any \(\ell \ne 2\), we can compute via the point
counting algorithm of~\cite[Section~4.2]{cubiccensus} that the
primitive Weil-polynomial \(f(t) = \det(F^*-t\mathrm{Id} |
H^4_\text{\'et,pr}(X_{\overline{\F}_2}, \Q_\ell(2)) ) \) of $X_{\F_2}$, is given by
\begin{align*}
f(t) = &{} t^{22} - t^{21} + t^{20} - \frac{3}{2}t^{19} + t^{18} - \frac{3}{2}t^{17} + \frac{3}{2}t^{16} - t^{15} + 2t^{14} - 2t^{13} + \frac{3}{2}t^{12}\\
&{} - 2t^{11} + \frac{3}{2}t^{10} - 2t^9 + 
    2t^8 - t^7 + \frac{3}{2}t^6 - \frac{3}{2}t^5 + t^4 -
    \frac{3}{2}t^3 + t^2 - t + 1.
\end{align*}
This Weil polynomial factors as \(f(t) = (t-1)^2g(t)\) for an
irreducible, noncyclotomic polynomial of degree 20 over $\Q$. It
follows from the Tate conjecture for cubic fourfolds over
 \(\F_2\)~(see \cite[Section~4.6]{cubiccensus}) that \( \rk
 \CH^2(X_{\overline{\F}_2}) = 3\). By the specialization
 theorem for Chow groups~\cite{fulton}, cf.\ \cite[\S2]{addington_auel}, one has 
\[
3=\rk \CH^2(X_{\overline{\F}_2})  \ge \rk \CH^2(X_{\overline{\Q}} ).
\]
Rigidity for Chow groups, together with the fact that they are
torsion-free, shows that extension of scalars yields an isometry
$\CH^2(X_{\overline{\Q}}) \isom \CH^2(X_{\C})$.  This implies that the
latter group has rank $3$, and that by
Proposition~\ref{prop:noadmissiblemarkings}, we have that \(X_\Q\) has
no twisted associated K3 surfaces (over \(\C\)).
 
It remains to check the that point counts of the K3 category of
\(X_{\F_2}\) are positive and exhibit expected growth. By the
formula~\eqref{eq:Xq}, it is enough to confirm the point count growth
properties over \(\F_{2^k}\), \(k =1, \ldots, 4\) (cf.\
\cite{kedlaya_sutherland-census}), but we present a bit more:

 \begin{center}
\begin{tabular}{c|ccccccccccc}
\(k\) &1&2&3&4&5&6&7&8&9&10&11\\
\(|\AA_{X_{\F_2}}(\F_{2^k})|\)&
7&
13&
85&
273&
1137 &
4081 &
16289 &
64001 &
264001 &
1052673 &
4196353\\
\end{tabular}
\end{center}

Hence the point counts of \(\AA_{X_{\F_2}}\) satisfy (3).
\end{proof}

\begin{remark} To find the example we provided above, we used our complete tabulation of zeta functions of cubic fourfolds over \(\F_2\) from~\cite{cubiccensus} to find a cubic with the desired algebraic and geometric rank of 3 and with discriminant $68$ according to the (conjectural) Artin-Tate formula for fourfolds~~\cite{lichtenbaum}. We then verified that this cubic \(X'\) over \(\F_2\) contains a configuration of a Veronese and a cubic scroll intersecting in two points. Finally, we found a smooth lift \(X/\Q\) containing the a Veronese and a scroll in the right configuration by first lifting the Veronese and scroll from \(\F_2\) first, and then searching for lifts of \(X'\) containing these two surfaces. 
\end{remark}

\begin{remark}\;~\;
\begin{enumerate} 
\item  Since one expects something like
Proposition~\ref{prop:noadmissiblemarkings} to hold in positive
characteristic, it is conceivable that the cubic fourfolds in
Theorem~\ref{thm:specialexample} have nongeometric K3 category over
\(\overline{\F}_2\).

\item However, the Weil polynomial of \(\AA_{X_{\F_2}}\) are of K3
type, and so appear on Kedlaya and Sutherland's list of potential Weil
polynomials of K3 surfaces over \(\F_2\)~\cite[Computation
  3(c)]{kedlaya_sutherland-census}.  This suggests that nongeometric
  K3 categories may have point counts identical to those of K3
  surfaces, or that additional restrictions on the Weil polynomials of
  K3 surfaces are needed for a complete Honda--Tate theory for K3 surfaces.
\end{enumerate}
\end{remark}

\end{document}